\documentclass[11pt]{amsart}

\usepackage{amssymb,amsmath}
\usepackage{amsthm}
\usepackage{amsfonts,newlfont}
\usepackage[dvips]{graphicx}

\usepackage{enumerate}
\usepackage{euscript}

\newcommand{\F}{\mathcal{F}}
\newcommand{\Or}{\mathcal{O}}
\newcommand{\T}{\mathbb{T}}
\newcommand{\R}{\mathbb{R}}
\newcommand{\Z}{\mathbb{Z}}
\newcommand{\C}{\mathbb{C}}
\newcommand{\Rm}{\mathbb R^m}

\renewcommand{\Re}{\mathfrak{Re}}

\newcommand{\Zk}{\mathbb Z^k}

\newcommand{\Rk}{\mathbb R^k}

\renewcommand{\epsilon}{\varepsilon}

\def \a{\alpha}
\def\m{\mathbf m}
\def\n{\mathbf n}
\def\t{\mathbf t}
\def\s{\mathbf s}
\def\g{\mathbf g}

\def \an{{\alpha (\mathbf n)}}
\def \at{{\alpha (\mathbf t)}}

\def \as{{\alpha (\mathbf s)}}
\def \ao{{\alpha _0}}

\def \aot{{\alpha _0 (\mathbf t)}}
\def \aos{{\alpha _0 (\mathbf s)}}
\def \bt{{\beta (\mathbf t)}}

\def \w{\cal W}
\def \vt{\widetilde {\cal V}}

\newcommand{\Ce}{C^{1+\theta}}

\newcommand{\e}{\varepsilon}

\def \e{\varepsilon}

\def \N{\cal N}

\newtheorem{theorem}{Theorem}[section]

\newtheorem{proposition}[theorem]{Proposition}
\newtheorem{corollary}[theorem]{Corollary}
 \newtheorem*{conjecture}{Conjecture}
\newtheorem{lemma}[theorem]{Lemma}

  \theoremstyle{definition}
        \newtheorem{remark}{Remark}
        
        \newtheorem{definition}{Definition}

 \numberwithin{equation}{section}
\begin{document}
\title[Measure and cocycle rigidity]{Measure and cocycle rigidity for certain non-uniformly hyperbolic  actions of higher rank abelian groups}

\author[Anatole Katok  and Federico Rodriguez Hertz]{Anatole Katok *) and Federico Rodriguez Hertz **)}
  \address{The Pennsylvania State University, University Park, PA}

\email{katok\_a@math.psu.edu}
\address{IMERL, Montevideo, Uruguay}
\email{frhertz@fing.edu.uy}
\date{\today}
\thanks {*)  Based on research  supported by NSF grants DMS-0505539  and  DMS-0803880}
\thanks{**) Partially supported by the Center for Dynamics and Geometry at Penn State.}
\maketitle

\begin{abstract} We prove absolute continuity of ``high entropy'' hyperbolic  invariant measures  for smooth actions of higher rank abelian groups assuming that  there are no proportional Lyapunov exponents.  For actions on tori and infranilmanifolds  existence of an absolutely continuous invariant measure of this kind  is obtained for  actions whose  elements are homotopic to those of an action by hyperbolic  automorphisms with  no  multiple or proportional Lyapunov exponents. In the latter case   a form of rigidity    is  proved  for  certain natural classes   of cocycles over the action. 

\end{abstract}
\bigskip

\section{Introduction}
In this paper we  continue the program of studying hyperbolic measures for actions of higher rank abelian groups  first alluded to in \cite[Part II]{kk-01}  and started in earnest \cite{kk, KRH, KKRH}.
We refer to those papers  for basic definitions and  standard facts concerning those actions.

Specifically we extend  some of the principal results of \cite{kk, KRH, KKRH} from maximal rank actions ($\Zk$ actions on  $k+1$-dimensional manifolds and $\Rk$ actions on $2k+1$-dimensional manifolds,  $k\ge 2$)  to a class of actions where dimension and rank are not related,  except of the  standard assumption of  rank being  at least 2.  Thus we  partially realize the ``Low rank and high dimension'' program of  \cite[Section 8.3.]{KKRH}.  While we  use the general methods and some specific results  from the previous papers as well as   heavy machinery of smooth ergodic theory, we introduce  three important new ingredients that make these advances possible. These new elements are:
\begin{itemize}
\item The  {\em holonomy invariance} that appears in the proof of  Theorem~\ref{graphargument}.
It is an extension to the general non-linear and  non-uniformly hyperbolic actions  of certain arguments which appeared in \cite{kk-01} for  the study of invariant measures for linear actions on the torus. \smallskip

\item  New {\em entropy inequality}, Lemma~\ref{Newentropy}, that is crucial in  the proof of Theorem~\ref{toralcase}, by allowing to show that all Lyapunov hyperplanes for the linear action  persist for the non-linear one.\smallskip

\item The {\em  index argument} in the uniqueness  proof  (Section~\ref{SS:uniqueness}, Lemma~\ref{indexargument}) that replaces the argument  in \cite{KRH} that depends on existence of elements with codimension one stable foliations.

\end{itemize}

Our second goal is to  prove cocycle rigidity  for actions on the torus
satisfying our  measure rigidity results.  For the case of maximal rank
actions those results  have been announced in
\cite{KKRH-ERA}.

In this paper  we  restrict  ourselves to the  case where all Lyapunov exponents are simple and  there are no proportional Lyapunov exponents.
This allows to avoid  extra technical complications  that appear  in the presence of multiple or proportional exponents. In  this situation coarse  Lyapunov  foliations are one-dimensional  and invariant geometric structures on their leaves  are affine. Our approach extends to certain cases where multiple or positively proportional  Lyapunov exponents are allowed (totally non-symplectic condition, TNS for short). In this case  one needs to use  a version of the theory of non-stationary normal forms (see \cite[Section 6] {kk-01}) to produce  invariant geometric structures on the leaves of  coarse  Lyapunov foliations. Those structures may be more complicated  than affine  if there are resonances between Lyapunov exponents. We discuss this more general situation in the last section. Detailed treatment will appear in a separate paper.
\section{Formulation of  results}\label{s:measurerigidity}

Let $\a$ be an $\Rk,\,k\ge 2$, $C^{1+\theta},\,(\theta> 0)$ action on an $n$-dimensional
manifold $M$ and $\mu$ be an invariant ergodic  measure for $\a$. Let
$\chi_1,\dots,\chi_n;\,\, \Rk \to\R$ be the Lyapunov exponents (linear functionals)
associated to $\mu$.
Recall that an ergodic invariant measure $\mu$ for a smooth locally free $\Rk$
action $\a$ is called {\em hyperbolic} if all nontrivial Lyapunov
exponents  $\chi_i$, $i=1, \dots, l$, are nonzero linear functionals
on $\Rk$. Kernels of  non-zero Lyapunov exponents are called {\em Lyapunov hyperplanes}. Vectors in $\R^k$  which do not lie on any of the Lyapunov hyperplanes are called {\em regular}.  Connected components of the sets of regular vectors are called {\em Weyl chambers}.

 Recall that in the absence of positively proportional Lyapunov exponents
every Lyapunov distribution $E_i$  integrates to an invariant family of smooth manifolds $\w^i$  defined $\mu$ a.e. which is customarily called the Lyapunov foliation. Leaves of of those foliations are intersections of stables manifolds of properly chosen elements of the action.
 See Section~\ref{section-prelim} for  details.
 
\subsection{Strongly simple actions}\label{sbs:stronglysimple}

\begin{definition}We say that $(\a,\mu)$ (or simply $\a$ if $\mu$ is understood) is
{\em strongly simple} if coarse Lyapunov distributions $E^i$  are one dimensional. Equivalently,  all Lyapunov exponents are simple and there are no proportional Lyapunov exponents.

We
say that $(\a,\mu)$ satisfies  the {\em full entropy condition} if the  entropy function is
not differentiable at Lyapunov hyperplanes.
\end{definition}

\begin{theorem}\label{stronglysimple}
Let $\mu$ be an ergodic  invariant measure for a strongly simple action $\a$. Then for any element $\t$ of the action such that the entropy $h_\mu(\t)>0$, there exists a Lyapunov exponent $\chi$ such that $\chi(\t)<0$ and  conditional
measures on the Lyapunov foliation $\w$ corresponding to $\chi$  are  equivalent to Lebesgue measure.
\end{theorem}
Since all  Lyapunov exponents change sign for an inverse transformation while the entropy remains the same, Theorem~\ref{stronglysimple} immediately implies the following result.

\begin{corollary} Let $\mu$ be an ergodic  invariant measure for a strongly simple action $\a$. If $h_\mu(\t)>0$ for some $\t\in\Rk$ then there are  at least two Lyapunov foliations such that the corresponding conditional measures are equivalent to Lebesgue measure.
\end{corollary}

It is  probable   that one can strengthen Theorem~\ref{stronglysimple} 
in the following way. 

\begin{conjecture}  Conditional measures on unstable  manifolds of the action elements are  equivalent to Lebesgue measures 
on certain smooth submanifolds that are obtained by integrating  those Lyapunov foliations for which conditional measures are Lebesgue. 
\end{conjecture}

The full entropy condition leads to  a  stronger assertion.

\begin{theorem}\label{general}
Let $\mu$ be an ergodic  invariant measure for an  action $\a$
Assume that
\begin{enumerate}
\item $(\a, \mu)$ is strongly simple;
\item $(\a, \mu)$  satisfies  the full entropy condition.
\end{enumerate} Then $\mu$
is absolutely continuous (with respect to the smooth measure class on $M$).
\end{theorem}

\subsection{Actions on tori and nilmanifolds}\label{sbs:tori-nil}

Let $N$ be a simply connected nilpotent Lie group and $A$  a  group of affine transformations of  $N$ acting freely  that contains  a finite index subgroup  $\Gamma$ of  translations  that is a lattice in $N$.  Then the  orbit space $N/A$ is a compact manifold that is called an {\em infranilmanifold}. An automorphism of $N$ that maps orbits of $A$ onto orbits of $A$ generates a  diffeomorphism of $N/A$ that is called  an infranilmanifold automorphism. 

An action $\a_0$ of  $\Zk$ by automorphisms  of an infranilmanifold $M$  is an Anosov action 
if induced linear action on the Lie algebra $\frak N$  of $N$  has non-zero Lyapunov exponents.  

Now let $\alpha$  be an action of $\Zk$ by diffeomorphisms of $M$  such that its elements are homotopic to  elements 
of an Anosov action by automorphisms. We will say that $\a$ has {\em homotopy data}  $\a_0$.
Recall that there is a  unique continuous map $h : M \to M$
homotopic to identity  such that
$$
h \circ \a=\a_0 \circ h.
$$
This map is customarily called the  semi-conjugacy between $\a$ and $\a_0$.

Let us call an  $\a$-invariant Borel  probability  measure $\mu$  {\em large} if the push-froward 
$h_*\mu$ is Haar measure.  The following theorem is the extension of  (Theorems 1.3, and 1.7) from \cite{kk}  from  actions on a torus with Cartan homotopy data to our case of actions on infranilmanifolds with strongly simple homotopy data. 
\begin{theorem}\label{toralcase} Let $\a_0$ be a strongly simple  Anosov action  of $\Zk$ by automorphisms of an infranilmanifold and let $\a$ be a smooth action
with homotopy data $\a_0$. Let $\mu$ be an  ergodic  large invariant measure  $\mu$ for $\a$. Then:
\begin{enumerate}
\item $\mu$ is absolutely continuous.
\item Lyapunov characteristic exponents of the action 
$\a$ with respect to  $\mu$ are equal to the Lyapunov 
characteristic exponents of the  action $\ao$. 
\end{enumerate}
\end{theorem}

\begin{conjecture}
Under  the assumptions of Theorem~\ref{toralcase}  large invariant measure is unique.
\end{conjecture}

 We can prove this under a stronger  assumption on $\a_0$ that forces the  group $N$  be abelian and  hence $M$  to have a torus as a finite cover. 

A relation  between different Lyapunov exponents   $\chi_1,\dots,\chi_l$ of a $\Z^k$ action 
of the form $\chi_1=\sum_{i=2}^l m_i\chi_i$ where $m_2,\dots m_l$ are positive integers is called a {\em resonance}.  A resonance with $l=2$,  is  called a {\em double resonance}

In this setting we may assume  that $N=\Rm$ and $\Gamma=\Z^m$. The factor of the torus $\T^m=\Rm/\Z^m$  by  finite group of fixed point free affine maps is called an {\em infratorus}.  The following theorem is a  generalization of the main result (Corollary 2.2) of
\cite{KRH}.

\begin{theorem}\label{toralcase2}
Let $\a_0$ be a linear strongly simple  $\Z^k$ action  without resonances on an infratorus. 
Any action  
 $\a$ with homotopy data
$\a_0$ has unique large invariant measure $\mu$.  Furthermore, the semiconjugacy  $h$  is bijective  $\mu$ a.e.  and effects a measurable isomorphism between $(\a,\mu) $ and $\a_0$ with Haar measure.  \end{theorem}

\begin{remark} Difference between the Cartan condition  of \cite{kk} and  our strongly simple condition is quite considerable. Cartan actions are  maximal rank actions on the torus by hyperbolic automorphisms; thus the rank is equal to the  dimension minus one. On the other hand,  strongly simple actions may have any rank starting from two: for example,  the restriction os a Cartan action to any $Z^2$ subgroup is strongly simple. Also  strongly simple actions may be reducible; e.g. the  product of two Cartan actions  is usually strongly simple. 
\end{remark}

\begin{remark} Actions by automorphisms of non-abelian  simply connected Lie
groups always have resonances that appear from the  non-trivial bracket relations in the Lie algebra. While  the  easiest standard examples have double
resonances there are strongly  simple actions  on some compact nilmanifolds. 
\end{remark}

Generally speaking, cocycle rigidity means that one-cocycles of certain regularity  over a group  action
are cohomologous  to constant cocycles via transfer functions  of certain (often lower) regularity.
Cocycle rigidity is  prevalent  in hyperbolic and partially hyperbolic actions of higher rank abelian groups, see e.g. \cite{KNT,KS1,KS2,DK-II}.  The method of \cite{KNT} for uniformly hyperbolic actions  satisfying   TNS
condition (no negatively proportional Lyapunov exponents)  can be used in the non-uniformly hyperbolic case and is in particular applicable in the settings considered in the present paper.

One can apply  this method to smooth or H\"older continuous cocycles but there are  reasons
to consider  two broader classes of cocycles which are in general defined only almost everywhere with respect to a hyperbolic absolutely
continuous invariant measure:
\smallskip

\noindent(i) {\em Lyapunov H\"older} cocycles: H\"older continuous with respect to a properly defined
{\em Lyapunov metric} which is equivalent to a smooth metric on Pesin sets and  changes slowly along the orbits, and
\smallskip

\noindent (ii) {\em Lyapunov smooth} cocycles: smooth along invariant foliations at the  points of Pesin sets with a similar slow change condition.
\smallskip

See Section~\ref{s: cocycles} for a more detailed formal description of those classes of cocycles.
Notice that the most important intrinsically defined cocycles, the  logarithms of  the Jacobians along  invariant foliations (Lyapunov, stable, and likewise), are Lyapunov smooth.

\begin{theorem}\label{thm-cocycles} For any action $\a$ of $\Zk$ on an infratorus as in Theorem~\ref{toralcase2} any Lyapunov H\"older  (corr. Lyapunov smooth) cocycle  is cohomologous to a constant cocycle via a Lyapunov H\"older (corr. Lyapunov smooth) transfer function.
\end{theorem}

\begin{remark} Assertion of this theorem is likely to hold in  the more general setting of Theorem~\ref{toralcase}. The difficulties in the proof are of  technical  nature and 
have to do  with a proper application of  a version of the Hopf argument in the resonance case when  invariant  geometric structures on stable foliations are not affine. 
\end{remark}
\subsection{Technical  results and overall structure of proofs}

\subsubsection{}The starting point in the proofs of all three  results formulated above,  Theorems~\ref{stronglysimple}, \ref{general} and  \ref{toralcase} is the Main technical
Theorem we proved in \cite[Theorem 4.1]{KKRH} joint with B. Kalinin.
So, let us recall it.

\begin{theorem}  \label{tech}
Let $\mu$ be a hyperbolic ergodic invariant measure for a locally
free $\Ce,\, \theta >0$, action $\a$ of $\Rk,\,\,k\ge 2$, on a
compact smooth manifold $M$. Suppose that a Lyapunov exponent $\chi$
is simple and there are no other exponents proportional to $\chi$.
Let $E$ be the one-dimensional Lyapunov distribution corresponding
to the exponent $\chi$ and $\w$ the corresponding
 Lyapunov foliation.

Then
conditional measures of $\mu$ on $\w$ are either atomic a.e. or
equivalent to Lebesgue measure a.e.
\end{theorem}

\subsubsection{}Let $\a$ and $\chi$ be as in the hypothesis of technical Theorem
\ref{tech}. Let us fix a generic singular element $\t\in\Rk$, i.e. an element  such
that $\chi(\t)=0$ but $\chi_j(\t)\neq 0$ for any other Lyapunov
exponent, and an element $\s$ close to $\t$ such that $\chi(\s)<0$
but it is still the biggest negative Lyapunov exponent for $\s$.
Then we have that $\w^s_{\at}\subset \w^s_{\as}$ and in fact
$E^s_{\as}=E^s_{\at}\oplus E_\chi$.

The second  principal technical  result  that  appears in the proofs of Theorems~\ref{general} directly  and \ref{toralcase}  (through Theorem~\ref{nondifofentropy}) is new.   It shows that  if we have atomic
conditional measures along some $\w^i$ direction then conditional
measures along stable manifolds containing $\w^i$ direction fit in a
lower dimensional submanifold. 

\begin{theorem}\label{graphargument}
If $\chi, \t$ and $\s$ are as above and the conditional measure
along $\w$ is atomic for a.e. point then the conditional measure
along $\w^s_{\as}(x)$ has its support inside $\w^s_{\at}(x)$ for
a.e. $x$.
\end{theorem}

While the proof of Theorem~\ref{graphargument} does not  use Theorem~\ref{tech} directly it relies on the the principal technical construction that appears in the proof of the latter result: the synchronizing time change, see Section~\ref{sbs:synchronizing}. 

Theorems~\ref{tech} and \ref{graphargument} are also used in the  proof of the following result  that  is used in the proof of Theorem~\ref{general}.

\begin{theorem}\label{nondifofentropy}
Let $\chi$ be as in Theorem \ref{tech} then conditional measure
along $\w$ is Lebesgue if  entropy function is not
differentiable at the Lyapunov hyperplane $\ker\chi$.
\end{theorem}

\subsubsection{} In addition to innovations listed  in the Introduction other additional  major ingredients that  appear  in the proofs are:

(i) Ledrappier-Young entropy formula reviewed in Section~\ref{lyentropyformula} and used in proofs of Theorems~\ref{nondifofentropy} and \ref{general} and 

(ii)  
an extension of H.Hu's result   on linearity of the entropy functions inside Weyl chambers, see \ref{Hu-extension}.  
It  is used in the proof of Theorem~\ref{nondifofentropy} to  treat a  case when a Lyapunov exponent is proportional  to the difference of two  other,  that  may appear in strongly simple actions. 
\begin{remark} It is  probable that the converse to the statement of Theorem~\ref{nondifofentropy} is also true.  This would follow from an extension of the Ledrappier-Young formula, see  statement $(\frak A)$ in Section~\ref{lyentropyformula}.
\end{remark}

\section{Preliminaries}\label{section-prelim}
\subsection{Lyapunov exponents and  Pesin sets}
For a smooth $\Rk$ action $\a$ on a manifold $M$ and an element $\t\in\Rk$ we
denote the corresponding diffeomorphism of $M$ by $\a(\t)$. Sometimes we will
omit $\a$ and write, for example, $\t x$ in place of $\a(\t)x$ and $D \t$ in place of
$D\a(\t)$ for the derivative of $\a(\t)x$.

 \begin{proposition} \label{MET}
Let $\a$  be a  locally free $\Ce,\, \theta >0$, action of $\R^k$ on a manifold $M$
preserving an ergodic invariant measure $\mu$.
There are linear functionals $\chi_i$, $i=1, \dots, l$, on $\R^k$  and an $\alpha$-invariant
measurable splitting  of the tangent bundle $TM$, called the {\em Lyapunov decomposition}, (or sometimes the  {\em Oseledets decomposition}),
$TM = T\Or \oplus \bigoplus_{i=1}^{l} E_{i}$
over a set of full measure $\Re$, where $T\Or$ is the distribution tangent to the $\Rk$
orbits, such that for any $\t \in \Rk$ and any  nonzero vector  $ v \in E_{i}$ the Lyapunov
exponent of $v$ is equal to $\chi_i (\t)$, i.e.
 $$
   \lim _{n \rightarrow \pm \infty }
   n^{-1} \log \| D(n\t) \, v \| = \chi_i (\t),
 $$
where $\| \cdot \|$ is any continuous norm on $TM$.
Any point $x\in \Re$ is called a {\em regular point}.

Furthermore, for any $\e > 0$ there exist positive measurable functions $C_\e (x)$
and $K_\e (x)$ such that for all $x \in \Re$, $v \in E_i(x)$, $\t \in \Rk$, and $i=1, \dots, l$,
 \begin{enumerate}

\item $C^{-1}_\e (x) e^{\chi_i(\t)-\frac{1}{2} \e\|\t\|} \| v\|\le
\|D \t \,v \|  \le  C_\e (x) e^{\chi_i(\t)+\frac{1}{2} \e\|\t\|} \| v\|$;

\item Angles $\angle (E_{i}(x),T\Or) \ge K_\e (x)$ and
         $\angle (E_{i}(x),E_{j}(x)) \ge K_\e (x)$, $i \not= j$;

\item $C_\e (\t x) \le C_\e (x) e^{\e \| \t \|}$
 and $K_\e (\t x) \ge K_\e (x) e^{-\e \| \t \|}$.

\end{enumerate}

\end{proposition}

The stable and unstable distributions $E^s _{\at}$ and $E^u _{\at}$ of an element $\a(\t)$ are defined as
the sums of the Lyapunov distributions corresponding to
the negative and the positive Lyapunov exponents for $\at$ respectively.  Notice that stable and unstable distributions are the same within a  Weyl chamber, and conversely,  the set of    vectors  with  given stable an  unstable distributions, if non-empty, is a Weyl chamber. A minimal non-zero intersection of stable distributions  for various elements of the action  is called a {\em  coarse Lyapunov  distribution}. Equivalently, any coarse Lyapunov distribution is the sum of Lyapunov distributions  corresponding
all Lyapunov exponents  positively proportional to each other. Notice that
in the absence of positively proportional Lyapunov exponents, in particular, for the strongly simple case considered in this paper, coarse Lyapunov distributions coincide with Lyapunov distributions.

\subsection{Invariant manifolds and affine structures}
We will use  standard material on invariant  manifolds corresponding to the
negative and positive Lyapunov exponents (stable and unstable manifolds) for
$\Ce$ measure preserving diffeomorphisms of compact manifolds, see for
example \cite[Chapter 4]{BP}. In particular, stable distributions and hence their transversal intersections are always H\"older continuous (see, for example,
\cite{BPbook}).  Here is a summary of some of those results adapted to the case of an $\Rk$ action.
 \begin{proposition} \label{Holderness}

Let $\a$ be a $\Ce,\, \theta >0$ action of  $\R^k$
as in Proposition \ref{MET}.
Suppose that a Lyapunov distribution $E$ is the intersection of the stable distributions
of some elements of the action. Then $E$ is H\"older continuous on any Pesin set
\begin {equation} \label{Pesinset}
\Re_{\e}^l =\{ x \in \Re  : \; C_\e (x) \le l, K_\e (x) \ge l^{-1} \}
\end{equation}
with H\"older constant which depends on $l$ and H\"older exponent $\delta>0$ which
depends on the action $\a$ only.

Furthermore on those sets the size of local stable manifolds for any element of $\a$  is bounded  away from below.
 \end{proposition}

We will denote by $\w_\at^s(x)$ the (global) stable manifold for $\at$ at a
regular point $x$. This manifold is an immersed Euclidean space tangent to the
stable distribution $E_\at^s$. The unstable manifold $\w_\at^u(x)$ is defined as
the stable one for $\a(-\t)$ and thus have similar properties.
Local stable/unstable  manifolds will be denoted by  $W_\at^s(x)$ and  $W_\at^u(x)$ correspondingly.

Intersections of stable manifolds for different elements of the action  are integral manifolds for coarse Lyapunov distributions and they form a
{\em coarse Lyapunov foliations}. While in general Lyapunov foliations may not be uniquely integrable this is obviously the  case  in the absence  of positively proportional Lyapunov exponents. Hence everywhere in this paper we will talk about {\em Lyapunov foliations}.
These foliations  are defined for any Lyapunov distribution $E$ as in Proposition
\ref{Holderness}. We will denote the Lyapunov foliation corresponding to the exponent $\chi$ by  $\w$ and its  local leaf  at a regular point $x$ by $W(x)$.

As a  comment on terminology  let us emphasize that
it is customary to use words ``distributions'' and ``foliations'' in this setting although
these objects are correspondingly measurable  families of tangent spaces defined
a.e. and measurable families of smooth manifolds which fill a set of full measure.

Let us also recall the existence of affine structures. Let $\a$ be
an action as in Theorem \ref{tech}. The following proposition
provides $\a$-invariant affine parameters on the leaves of the
Lyapunov foliation $\w$.

\begin{proposition}\cite[Proposition 3.1., Remark 5]{kk}\label{affineparameters}
There exist a unique family of $C^{1+\theta}$ smooth $\a$-invariant
affine parameters on the leaves $\w(x)$. Moreover, they depend
uniformly continuously in $\C^{1+\theta}$ topology on $x$ in any Pesin
set.
\end{proposition}

\subsection{Lyapunov metrics and synchronizing time change}\label{sbs:synchronizing}
We fix a smooth Riemannian metric $<\cdot,\cdot>$ on $M$. Given $\e>0$
and a regular point $x \in M$ we define {\em the standard $\e$--Lyapunov scalar
product (or metric)} $<\cdot,\cdot>_{x,\e}$ as follows. For any $u,v\in E(x)$ we define
\begin{equation}\label{eqLyapunovmetric}
<u,v>_{x,\e}=\int_{\R^k}<(D \s)  u, (D \s)  v>
\exp(-2\chi(\s)-2\e\|\s\|) \, d\s.
\end{equation}

We shall need to use the time change we introduced with B. Kalinin
in \cite{KKRH} in the context of Theorem \ref{tech}. Let
$L=\ker\chi$, fix a vector $\mathbf w\in\R^k$ normal to $L$ with
$\chi(\mathbf w)=1$ and take $\e>0$ small such that $\e\|\mathbf
w\|$ is also small, in particular less than $1/2$.

\begin{proposition}\cite[Proposition 6.2, Proposition 6.3]{KKRH}\label{timechange}
For $\mu$-a.e. $x$ and any $\t\in\R^k$ there exists $g(x,\t)\in \R^k$
such that the function $\g(x,\t)=\t+g(x,\t)\mathbf w$ satisfies the
equality
$$\|D_x^E\a(\g(x,\t))\|_\e=e^{\chi(\t)}.$$
The function $g(x,\t)$ is measurable and is H\"older continuous on
Pesin sets, is $C^1$ in $\t$ and $|g(x,\t)|\leq 2\e\|\t\|$.
Moreover, the formula $\beta(\t,x)=\a(\g(x,\t))x$ defines an $\R^k$
action $\beta$ on $M$ which is a measurable time change of $\a$. The
action $\beta$ is measurable and continuous on Pesin sets for $\a$
and preserves a measure $\nu$ which is equivalent to $\mu$.
\end{proposition}

Now we describe  invariant ''foliations" for $\beta$ whose leaves are not  smooth  but still these objects  have properties close to
true invariant foliations for smooth actions. Let us denote with
$\N$ the orbit foliation of the one-parameter subgroup $\{r\mathbf
w\}$.

\begin{proposition}\label{invariantmanifoldtimechange}
For any element $\s\in\R^k$ there exist a stable "foliation"
$\tilde{\w}^s_{\beta(\s)}$ which is contracted, by $\beta(\s)$ and
invariant under the new action $\beta$. It consists of "leaves"
$\tilde{\w}^s_{\beta(\s)}(x)$ defined for every $x$. The "leaf"
$\tilde{\w}^s_{\beta(\s)}(x)$ is a measurable subset of the leaf
$(\N\oplus\w^s_{\a(\s)})(x)$ of the form
$$\tilde{\w}^s_{\beta(\s)}(x)=\{\a(\phi_x(y)\mathbf w)y:y\in\w^s_{\a(\s)}(x)\},$$
where $\phi_x:\w^s_{\a(\s)}(x)\to\R$ is an almost everywhere defined
measurable function. For $x$ in a Pesin set, the $\phi_x$ is
H\"older continuous on the intersection of this Pesin set with any
ball of fixed radius in $\w^s_{\a(\s)}(x)$ with H\"older exponent
$\gamma$ and constant which depends on the Pesin set and radius.
\end{proposition}

We will use the fact for any   $\s\in\R^k$  the partition into global  stable manifolds $\tilde{\w}^s_{\beta(\s)}(x)$ refines
the partition into ergodic components  of $\beta(\s)$.

\subsection{Ledrappier-Young entropy formula}\label{lyentropyformula}
 Let $f:M\to M$ be
$\Ce$ diffeomorphism and let $\mu$ be an ergodic invariant measure.
Let $\chi_1>\chi_2>\dots>\chi_r$ be its Lyapunov exponents and let
$TM=E_1\oplus\dots\oplus E_r$ be the corresponding Lyapunov decomposition. Let
$u=\max\{i:\chi_i>0\}$ and for $1\leq i\leq u$ let us define
$$V^i(x)=\{i\in M\limsup_{n\to\infty}\frac{1}{n}\log
d(f^{-n}(x),f^{-n}(y))\leq -\chi_i\}. $$ For a.e. $x$\,\,  $V^i(x)$ is a smooth
manifold  tangent to $\bigoplus_{j\leq i} E_j$ and we
have the flag $V^1\subset V^2\subset\dots\subset V^u$ with
$V^u=W^u$, the unstable manifold. We can build partitions $\xi^i$
subordinated to $V^i$ as the ones built in \cite{LY2} and consider
conditional measures $\mu_x^i$. Let  $B^i(x,\e)$  be the $\e$
ball in $V^i(x)$  centered in $x$ with respect to  the induced Riemannian metric.  Then 
$$\delta_i=\delta_i(f)=\lim_{\e\to 0}\frac{\log\mu_x^iB^i(x,\e)}{\log\e}$$
exists a.e. and does not depend on $x$. Moreover, calling
$\gamma_i=\gamma_i(f)=\delta_i-\delta_{i-1}$ we have the
Ledrappier-Young entropy formula (see \cite[Theorem C]{LY2})
\begin{equation} \label{LYF} h_\mu(f)=\sum_{1\leq j\leq u}\gamma_j\chi_j.\end{equation}
In fact, a more precise statement is true.
Given a measure $\mu$ and two measurable partitions $\alpha,\beta$, the conditional entropy is defined by:
$$H(\alpha|\beta)=-\int\log\mu_x^\beta(\alpha(x))d\mu.$$
 Let $T:X\to X$ be a measure preserving transformation. Given a measurable
partition $\alpha$ we define the entropy of $T$ w.r.t. $\alpha$ by
$h(\alpha,T)=H(T^{-1}\alpha|\alpha^+)$ where
$\alpha^+=\bigvee_{n\geq 0}T^n\alpha$.

We have for every $1\leq i\leq u$
$$h_\mu(\xi_i,f)=\sum_{1\leq j\leq i}\gamma_j\chi_j.$$

The following  addition to the Ledrappier-Young formula  is needed
 for the  proof of the  converse to Theorem\ref{nondifofentropy}. 
 \medskip

\noindent $(\frak A)${\em If in Ledrappier-Young formula \eqref{LYF}  $\gamma_i=0$, i.e.
$$h_\mu(\xi_i,f)=h_\mu(\xi_{i-1},f)$$ then the  conditional measure on an almost every leaf of
$V^i$ is supported on a single leaf of  $V^{i-1}$}
\medskip

While this statement looks completely natural and is likely to be true, 
existing arguments and constructions do not yield a proof.

\subsection{Linearity of entropies}\label{linearentrop}
Given two commuting diffeomorphisms $f$ and $g$, preserving a measure $\mu$, with coinciding unstable manifolds, in \cite{hu}, H. Hu built a partition subordinated to this unstable manifold which is increasing for both maps. The same construction may be carried out for partitions subordinated to simultaneous fastest directions, following the same lines, to get the following analogous of Proposition 8.1. in \cite{hu}.

Let $f$ and $g$ be two diffeomorphisms preserving a measure $\mu$. Assume both maps preserve a measurable bundle $F\subset TM$ such that Lyapunov exponents associated to $F$, for both $f$ and $g$, are larger than any other Lyapunov exponents. Hence we get that $F$ is tangent to a "foliation" V which is exactly the fastest foliation associated to the first $\dim(F)$ exponents. 

\begin{proposition}\label{huspartitions}
There is a measurable partition $\eta$ on $M$ with the following properties.
\begin{enumerate}
\item $\eta$ is subordinated to $V$.
\item $\eta$ is increasing for $f$ and $g$, i.e. $f\eta<\eta$ and $g\eta<\eta$.
\item $\bigvee_{n\geq 0}f^{-n}\eta$ and $\bigvee_{n\geq 0}g^{-n}\eta$ are the partitions into points. 
\item The biggest $\sigma-$algebra contained in $\bigcap_{n\geq 0}\bigcap_{k\geq 0}f^{-n}g^{-k}\eta$ is the $\sigma-$algebra of sets saturated by leafs of $V$.
\end{enumerate}
\end{proposition}

As always, $(3)$ and $(4)$ follows from $(1)$ and $(2)$. Also, as in \cite{LY}, we have that the entropy of $f$ and $g$ w.r.t. this partition does not depend on the partition as long as the partition is subordinated to $V$, i.e.  
\begin{lemma}
If $\eta$ and $\hat{\eta}$ are two partitions subordinated to $V$ as in Proposition \ref{huspartitions}, then $h_{\mu}(\eta,f)=h_{\mu}(\hat{\eta},f)$. The same holds for $f$, $g$ and $f\circ g$. 
\end{lemma}
Thus, we can call $h_\mu(V,f)=h_{\mu}(\eta,f)$.

This gives as the analogous of Proposition 9.1. in \cite{hu} which gives that entropy is linear.
\begin{proposition}\label{linearityent}
Let $f$, $g$, $\mu$, $V$ and $\eta$ be as above. Then $$h_{\mu}(V,f\circ g)=h_{\mu}(V,f)+h_{\mu}(V,g).$$ 
\end{proposition}
Since the proof is in two lines we repeat it here, let us write $fg$ for $f\circ g$. 
\begin{proof}
\begin{eqnarray*}
h_{\mu}(\eta,fg)&=&H(\eta|fg\eta)=H(\eta\vee g\eta|fg\eta)
=H(g\eta|fg\eta)+H(\eta|g\eta\vee fg\eta)\\
&=&H(\eta|f\eta)+H(\eta|g\eta)=h_{\mu}(\eta,f)+h_{\mu}(\eta,g)
\end{eqnarray*}
\end{proof}
\begin{corollary}\label{Hu-extension}If $C$ is the cone where $F$ is still the fastest bundle, then $$\s\to h_{\mu}(V,\as)$$ is linear for $\s\in C$.
\end{corollary}

\section{Proof of Theorems \ref{graphargument} and \ref{nondifofentropy}}
\subsection{Holonomy invariance of conditional
measures: A model case}\label{holonomyinvariance}

Before  considering the situation  that appears in  Theorem~\ref{graphargument} we shall discuss another case of holonomy invariance of conditional measures  along stable directions  that is of independent
interest and that has  a similar but simpler proof.

Let $f$ be a diffeomorphism preserving a measurable
 foliation $\F$ with smooth leaves. Assume also
that $Df|T\F$ is an isometry. Let $\mu$ be an $f$-invariant measure 
and  $\mu^{\F}_x$ be the  conditional measure along $\F(x)$. Recall that those measures are defined up to a scalar  multiple; we shall use the normalization
$\mu^{\F}_x(B^{\F}(x,1))=1$,  where $B^{\F}(x,1)$ is the  the ball  in the leaf $\F(x)$ centered
in $x$  of radius $1$.  
Given  a regular point  $x$ and $y\in W^s(x)$
let us call $h_{xy}:\F(x)\to\F(y)$ the holonomy along the stable
manifolds; if it is clear from the context we shall omit the
lower index $xy$.

\begin{proposition}\label{isometrycase}
Let $f$, $\F$ and $\mu$ be as above. Then for $\mu$-a.e. $x$ and
$\mu_x^s$-a.e. $y$ in $W^s(x)$ we have that
$h_*\mu^{\F}_x=\mu^{\F}_y$.
\end{proposition}

\begin{proof}
 If $x$ and $y$ are in the same leaf of $\F$  then
$\mu_x^\F=c_{xy}\mu_y^\F$ for some positive constant $c_{xy}$.
Also, by invariance of $\mu$ and $\F$ and that $f$ restricted to
$\F$-leafs is an isometry we get that $f_*\mu_x^\F=\mu_{f(x)}^\F$.

Consider now $\mu^{\F,1}_x=\mu^{\F}_x|B^\F(x,1)$ the restriction of
$\mu_x^\F$ to $B^\F(x,1)$. By invariance of conditional measures and our choice of normalizations   we have 
$f_*\mu_x^{\F,1}=\mu_{f(x)}^{\F,1}$. 

We will prove  that for $\mu$-a.e.
$x$ and for $\mu_x^s$-a.e. $y\in W^s(x)$
$$h_*\mu^{\F,1}_x=\mu^{\F,1}_y$$ 
Take a sequence $n_i\to\infty$ such that
$f^{n_i}(x)\to z$, and hence  $f^{n_i}(y)\to z$. Since $f$
restricted to the leafs of $\F$ is an isometry we may take a
subsequence such that $f^{n_i}|\F(x)$ converges uniformly on compact
sets to an isometry $g_{xz}:\F(x)\to\F(z)$. Similarly we have that
$f^{n_i}|\F(y)$ converges uniformly on compact sets to an isometry
$g_{yz}:\F(y)\to\F(z)$ and 
$$h=g_{yz}^{-1}\circ g_{xz}$$ since stable manifolds are contracted in the
future. Thus it is sufficient to prove that 
that $(g_{xz})_*\mu^{\F,1}_x=\mu^{\F,1}_z$ and similarly
$(g_{yz})_*\mu^{\F,1}_y=\mu^{\F,1}_z$. But we have that
$f^{n_i}_*\mu_x^{\F,1}=\mu_{f^{n_i}(x)}^{\F,1}$ and that $f^{n_i}\to
g_{xz}$ and $f^{n_i}(x)\to z$. Hence we have 
$f^{n_i}_*\mu_x^{\F,1}\to (g_{xz})_*\mu_x^{\F,1}$ and
$f^{n_i}_*\mu_y^{\F,1}\to (g_{yz})_*\mu_y^{\F,1}$. So we need to
prove that $$\mu_{f^{n_i}(x)}^{\F,1}\to\mu_z^{\F,1}\quad \text{and} \quad 
\mu_{f^{n_i}(y)}^{\F,1}\to\mu_z^{\F,1}.$$

To this end we need to use some kind of continuity of the map
$x\to\mu^{\F,1}_x$. This map is only measurable so we need to do
something to guarantee   some kind of continuity. We can apply Luzin's theorem and obtain 
 continuity on a compact  set of an arbitrary large measure, so we need
to pick the iterates (and  hence $z$) inside this set.  The problem is to pick
{\em the same}  iterates of $x$ and $y$  in this large set and
we will now explain how to achieve that. 

We have that $x\to\mu^{\F,1}_x$ is a measurable map and so we have
by Luzin's theorem an increasing sequence of compact sets $K_n$, $\mu(K_n)\to 1$,   such that the map restricted to $K_n$ is
continuous. Consider $\widetilde{\mathbf 1}_{K_n}$ the forward Birkhoff
average of ${\mathbf 1}_{K_n}$ the characteristic function of $K_n$. Take
$R_n$ the set of points where $\widetilde{\mathbf 1}_{K_n}>1/2$;  $\mu(R_n)\to 1$ since
$\mu(K_n)\to 1$.  
Since the partition into stable manifolds refines the
partition into ergodic components of $f$,  for $\mu$-a.e. point
$x\in R_n$, $\mu^s_x$ a.e. point in $W^s(x)$ is inside $R_n$. Take
one of these typical points $x\in R_n$ and $y\in W^s(x)\cap R_n$.
Let  $L(x)=\{n\geq 0: f^n(x)\in K_n\}$ and similarly $L(y)=\{n\geq
0: f^n(y)\in K_n\}$. We have  from the choice of $R_n$ and since  $x,y\in R^n$
$$\frac{\#( L(x)\cap [0,n])}{n}\to \widetilde\chi_{K_n}(x)>\frac{1}{2}$$
and the same  is true for $y$.  Hence  both sets  $L(x)$ and $L(y)$ have asymptotic  density greater 
than $1/2$  and hence they should intersect in a set of
positive asymptotic  density, in particular  $L(x)\cap L(y)$ is an infinite set. So
we take the sequence $n_i$ inside $L(x)\cap L(y)$ and we get then
that $f^{n_i}(x)$ and $f^{n_i}(y)$ are inside $K_n$ and hence their
limit $z$ is also in $K_n$. Now by continuity of $x\to\mu^{\F,1}_x$
restricted to $K_n$ we get that
$\mu_{f^{n_i}(x)}^{\F,1}\to\mu_z^{\F,1}$ and
$\mu_{f^{n_i}(y)}^{\F,1}\to\mu_z^{\F,1}$.
\end{proof}

\subsection{Proof of Theorem~\ref{graphargument}} Now let us return to our case. We want to prove the same invariance
by holonomy of the conditional measures in a  less uniform but more specialized 
setting.

Let $\a,\, \mu$ and $\chi$  satisfy  the assumptions  of technical
Theorem \ref{tech}. Let us fix a generic singular element
$\t\in\Rk$, i.e. an element  such that $\chi(\t)=0$ but
$\chi_j(\t)\neq 0$ for any other Lyapunov exponent, and an element
$\s$ close to $\t$ such that $\chi(\s)<0$ but it is still the
biggest negative Lyapunov exponent for $\s$. Then we have that
$\w^s_{\at}\subset \w^s_{\as}$ and in fact
$E^s_{\as}=E^s_{\at}\oplus E_\chi$. We have the invariant foliation $\w$ associated to $\chi$ tangent to $E_\chi$ and we
consider the conditional measures $\mu^{\w}$ associated to this foliation that we normalize  in a certain  convenient way. 
Given $x$ and
$y\in\w^s_\at(x)$ we define the holonomy map $h_{xy}:\w(x)\to\w(y)$
by sliding along $\w^s_\at$ manifolds; we omit the  lower index $xy$ if it is understood from the context.

Theorem~\ref{graphargument}  is an immediate corollary of the following holonomy invariance property of conditional measures. 

\begin{proposition}\label{holonomyinvariancenonuniform}
Let $\chi, \t$ and $\mu$ be as above. Then for $\mu$-a.e. $x$ and
$\mu^s_x$-a.e. $y\in \w^s_\at(x)$,  there is a  scalar measurable function $c(x,y)$ such
that $h_*\mu^{\w}_x=c(x,y)\mu^{\w}_y$ where $h$ is holonomy along
$\w^s_\at$.
\end{proposition}

\begin{remark} Both Propositions \ref{isometrycase} and \ref{holonomyinvariancenonuniform}   assert that  the system of conditional measures, defined affinely,  is holonomy invariant. The difference is that in the former case there is a normalization that makes   normalized conditional measures invariant. 
\end{remark}

\begin{proof}
We will argue as in the proof of Proposition \ref{isometrycase} but
we need to address the problem that   dynamics  of $\a(t)$ along $\w$ is not an
isometry  although the Lyapunov exponent along $\w$ is equal to zero.

By Proposition \ref{affineparameters} there is  a measurable $\a$-invariant  family of
affine parameters $H_x:\R\to \w(x)$. 
We normalize $\mu_x^{\w}$ is such a way that
$\mu_x^{\w}(H_x(-1,1))=1$ and define
$\mu_x^{\w,1}=\mu_x^{\w}|H_x(1-,1))$.

We shall use the time change $\beta$ introduced in Proposition
\ref{timechange}. Using Luzin's Theorem, let $K_n$ be an increasing
sequence of compact sets (which we take also inside Pesin sets for
$\a$)   $\nu(K_n)\to 1$ for the $\beta$-invariant measure $\nu$ (and hence $\mu(K_n)\to 1$) such that on the set 
$K_n$ 
\begin{enumerate}
\item  the time change is continuous, \item  the map $x\to H_x$ is
continuous (with $C^1(\R, M)$
 topology for affine structures), and \item the map $x\to \mu_x^{\w,1}$ is continuous with weak * topology  in measures on $[-1,1]$. 
\end{enumerate}

Take $f=\beta(t)$ and consider $\widetilde{\mathbf 1}_{K_n}$ the forward
Birkhoff average of the characteristic function ${\mathbf 1}_{K_n}$. Let as before  $R_n$ be  the set of points where
$\widetilde{\mathbf 1}_{K_n}>1/2$. Since $\nu(K_n)\to 1$
then $\nu(R_n)\to 1$ (and hence $\mu(R_n)\to 1$). Also, since the partition into $\tilde\w^s_{\bt}$
stable "manifolds" refines the partition into ergodic components
(see the remark after Proposition \ref{invariantmanifoldtimechange})
we have that for $\nu$-a.e. point $x$ in $R_n$, $\nu^s_x$ a.e. point
in $\tilde\w^s_{\bt}(x)$ is inside $R_n$. Take one of these typical
points $x\in R_n$ and $y\in\tilde\w^s_{\bt}(x)\cap R_n$. As in the
proof of Proposition \ref{holonomyinvariance} we can take a sequence
of iterates $n_i$ such that $f^{n_i}(x)$ and $f^{n_i}(y)$ are inside
$K_n$ and hence their limit $z$ is also in $K_n$. Now by continuity
of $x\to\mu^{\w,1}_x$ restricted to $K_n$ we get that
$\mu_{f^{n_i}(x)}^{\w,1}\to\mu_z^{\w,1}$ and
$\mu_{f^{n_i}(y)}^{\w,1}\to\mu_z^{\w,1}$.

On the other hand if we denote ${\mathbf a}_i=\g(x,n_i\t)$ and
${\mathbf b}_i=\g(y,n_i\t)$ we have that $\|D_x^E\a({\mathbf
a}_i)\|_\e=1$ and $\|D_y^E\a({\mathbf b}_i)\|_\e=1$. Hence, using
the affine parameters from  Proposition \ref{affineparameters} we obtain
\begin{eqnarray}\label{affineparametersformula}
\a({\mathbf a}_i)\circ H_x=H_{\a({\mathbf a}_i)(x)}\;\;\;\mbox{
and}\;\;\; \a({\mathbf b}_i)\circ H_y=H_{\a({\mathbf b}_i)(y)}.
\end{eqnarray}
We have that the holonomy $\tilde h:\w(x)\to\w(y)$
 along $\tilde\w_{\bt}^s$ equals
$$\lim_{i\to\infty}(\a({\mathbf b}_i)|\w(y))^{-1}\circ P_i\circ(\a({\mathbf a}_i)|\w(x))$$
where $P_i$ is a sequence of smooth maps from $\w(\a({\mathbf
a}_i(x))$ to $\w(\a({\mathbf b}_i(y))$  converging to the
identity. Using \eqref{affineparametersformula} and  property (2) above we obtain
$$\lim_{i\to\infty}\a({\mathbf a}_i)|\w(x)=\lim_{i\to\infty}H_{\a({\mathbf a}_i)(x)}\circ H_x^{-1}=H_z\circ
H_x^{-1}=:g_{xz}$$ since $\a({\mathbf a}_i)(x)\to z=$. Similarly
$$\lim_{i\to\infty}\a({\mathbf b}_i)|\w(y)=\lim_{i\to\infty}H_{\a({\mathbf b}_i)(y)}\circ H_y^{-1}=H_z\circ
H_y^{-1}=:g_{yz}$$ since $\a({\mathbf b}_i)(y)\to z$ also. So we get
that
$$\tilde h=g_{yz}^{-1}\circ g_{xz}=H_y\circ
H_x^{-1}.$$

Again using \eqref{affineparametersformula} and the definition of
$\mu^{\w,1}$ we get that $\a({\mathbf
a}_i)_*\mu^{\w,1}_x=\mu^{\w,1}_{\a({\mathbf a}_i)(x)}$ and
$\a({\mathbf b}_i)_*\mu^{\w,1}_y=\mu^{\w,1}_{\a({\mathbf b}_i)(y)}$.

So, finally, putting all together we get and sending $i$ to infinity
we get that $(g_{xz})_*\mu^{\w,1}_x=\mu^{\w,1}_z$ and
$(g_{yz})_*\mu^{\w,1}_y=\mu^{\w,1}_z$ which gives that $\tilde
h_*\mu^{\w,1}_x=\mu^{\w,1}_y$.

Now, since $\tilde\w_{\bt}^s(x)$ is a graph over $\w_{\at}^s(x)$ we
get that $\nu^s_x$-a.e. point correspond to $\mu^s_x$-a.e. point and
hence we get that for $\mu$-a.e. point $x$ and for $\mu^s_x$-a.e.
point $y\in\w^s_{\at}(x)$ we get that $\tilde y=\a(\phi_x(y)\mathbf
w)(y)$ is a typical point for $\nu_x^s$ in $\tilde\w_{\bt}(x)$ and
hence the holonomy $\tilde h:\w(x)\to\w(\tilde y)$ makes $\tilde
h_*\mu^{\w,1}_x=\mu^{\w,1}_{\tilde y}$. So the proof finishes since
we have that $h:\w(x)\to\w(y)$ equals $h=(\a(\phi_x(y)\mathbf
w))^{-1}\circ\tilde h$ and w have that $\mu^{\w}$ is $\a$ invariant
modulo multiplication by a constant.
\end{proof}

\subsection{Proof  of Theorem \ref{nondifofentropy}}
Take $\t$ and $\s$ as in Theorem \ref{graphargument}. 
Take now a neighborhood of $\t$ such that positive Lyapunov exponents other 
than $\chi$ in this neighborhood are all bigger than $\chi$.  For $\s$ in this 
neighborhood, we call the strong unstable foliation associated to the positive 
Lyapunov exponents different from $\chi$, $V^{u-1}$ that does not depend on 
$\s$. Pick $\s$ in this neighborhood and observe that if $\s$ is on one side of 
$\ker\chi$, where  $\chi(\s)>0$, (denote this side $L^+$) there is only one more 
positive Lyapunov exponent, i.e. $\chi$, and no new exponent appears to the
other side $L^-$, where $\chi(\s)<0$. So, for $\s\in L^-\cup\ker\chi$, 
$V^{u-1}=\w^u_{\as}$ and for $\s\in L^+$, $V^{u-1}\subsetneq \w^u_{\as}$.

By Corollary~\ref{Hu-extension} we see that the map $\s\to h(\xi_{u-1},\as)$ is linear  in this neighborhood.

On the other hand, by Ledrappier-Young entropy formula we have that for 
$\s\in L^+$, $$h_{\mu}(\as)=h_{\mu}(\xi_u,\as)=h(\xi_{u-1},\as)+\gamma_u\chi(\s)$$
where $\xi_u$ is any partition subordinated to $\w^u_{\as}$ and $\gamma_u$
does not depend on $\s$  by its definition and the assumption on $\s$, see the
definition of $\gamma_u$ in subsection \ref{lyentropyformula}.

Finally, for $\s\in L^-\cup\ker\chi$, $\s$ close to $\t$ we have that
$$h_{\mu}(\as)=h(\xi_{u-1},\as).$$ 
Hence we have that on $L^+$, $h_{\mu}(\as)$ is linear and on $L^-$, $h_{\mu}(\as)$ 
is also linear. So, in order that this $h_{\mu}$ be differentiable its is necessary and 
sufficient that  this two linear maps coincide. But since $\s\to h(\xi_{u-1},\as)$ is linear 
in the whole neighborhood, this is the same as asking that $\gamma_u=0$. Hence, 
$h_\mu$ is differentiable at $\t$ if and only if $\gamma_u=0$.

Now, if conditional measures along $\w$  are atomic, then we have by Theorem
\ref{graphargument} that conditional measure along $V^u$ is
supported in $V^{u-1}$ and hence $\delta^u=\delta^{u-1}$ which gives
$\gamma^u=0$ and hence $h_\mu$ is differentiable at $\t$. \qed

\section{Proof of Theorems \ref{stronglysimple} and \ref{general}}
\subsection{Proof of Theorem~\ref{stronglysimple}}
By Theorem \ref{tech} we know that conditional measures along
Lyapunov directions are either Lebesgue or atomic. We will show that
if all conditionals are atomic, then conditional measures along
stable manifolds of $\at$ are atomic   and hence  $h_\mu(\t)=0$.

Observe that there without  loss of generality we may  assume $\t$  to be  a
regular element since  for every $\t$ one can find a regular element
$\t'$ whose stable foliation contains that of $\t$.  Hence if $h_\mu(\t)>0$ then  $h_\mu(\t')>0$.

We shall build a sequence of nested sub-foliations
$$\vt_0\supset\vt_1\supset\dots\supset\vt_n$$
where $\w^s_{\at}=\vt_0$, $\vt_n(x)=\{x\}$ and each $\vt_i$ is
either equal to $\vt_{i-1}$ or has one less dimension for a.e. $x$.
We shall prove also that conditional measures on the leaves  $\vt_i$ are
supported by single leaves of  $\vt_{i+1}$ and the theorem will follow.

Since $\t$ is regular it belongs to a Weyl chamber that we denote by $C_0$.  Let $\gamma$ be
a curve which begins at $\t$,  passes through every Lyapunov hyperplane and  crosses
each Lyapunov hyperplane only once at a point that does not lie on any other Lyapunov hyperplane. An examples of such a curve is the half-circle  in
a two-dimensional plane  through $\t$ that is  in general position, i.e. intersects all Lyapunov hyperplanes along different lines.

Let us number $C_0, C_1,\dots
C_n$ the Weyl chambers and $\chi_1,\dots,\chi_n$ the Lyapunov
exponents, in the order they appear.

The stable foliation does not change within a  Weyl chamber so
we shall denote $\w^s_{C_i}$ the stable foliation associated to this
Weyl chamber, also the sign of a Lyapunov exponent does not change
so we may denote this sign by  $\chi_j(C_i)$. So, we define
$\vt_0=\w^s_{C_0}$ and $\vt_i=\vt_{i-1}\cap\w^s_{C_i}$. Clearly
$\vt_i\subset\vt_{i-1}$ and $\vt_i$ is a nice foliation since it is
an intersection of stable foliations.

When passing from $C_{i-1}$ to $C_i$ $\chi_i$ is  the only Lyapunov exponent
that changes sign. So, if $\chi_i(C_i)<0$ then
$\chi_i(C_{i-1})>0$ so that $\w^s_{C_i}\supset\w^s_{C_{i-1}}$ and
hence $\vt_i=\vt_{i-1}$. On the other hand, if $\chi_i(C_i)>0$ then
$\chi_i(C_{i-1})<0$ and hence $\w^s_{C_i}\subsetneq\w^s_{C_{i-1}}$
and in this case $\vt_i\subsetneq\vt_{i-1}$, in fact $\w^i$ in no
more in $\vt_i$, i.e. $\w^i(x)\cap\vt_i(x)=\{x\}$. Moreover, if we
take an element inside $C_{i-1}$ but close to $\ker\chi_i$ we have
that $\w^i$ is the slowest direction in $\w^s_{C_{i-1}}$ while
$\vt_i$ is inside the fast direction in $\w^s_{C_{i-1}}$ (which is
exactly $\w^s_{C_i}$).

Let us fix measurable partitions $\eta_i$ subordinated to $\w^s_{C_i}$
such that  elements are open subsets of the leaves  $\mod 0$, i.e.  the conditional measures of the boundaries are equal to zero.

We shall chose  those partitions  in such a way that if
$\w^s_{C_i}\supset\w^s_{C_{i-1}}$ then $\eta_i<\eta_{i-1}$ and if
$\w^s_{C_i}\subset\w^s_{C_{i-1}}$ then $\eta_i>\eta_{i-1}$. Let
$\xi_0=\eta_0$ and define inductively the measurable partitions
$\xi_i=\xi_{i-1}\vee\eta_i$. It is sufficient  to prove that
$\mu_{\xi_0}^x(\xi_n(x))>0$ for a.e. $x$ since by construction
$\xi_n=\e$. To this end we shall argue inductively and prove that
$\mu_{\xi_{i-1}}^x(\xi_i(x))>0$ for a.e. $x$. Once we know this, we
have that for any measurable set $A$
$$\mu_{\xi_i}^x(A)=\frac{\mu_{\xi_{i-1}}^x(A\cap\xi_i(x))}{\mu_{\xi_{i-1}}^x(\xi_i(x))},$$
and hence
$$\mu_{\xi_{i-1}}^x(\xi_n(x))=\mu_{\xi_i}^x(\xi_n(x))\mu_{\xi_{i-1}}^x(\xi_i(x)),$$
which gives
$$\mu_{\xi_0}^x(\xi_n(x))=\prod_{i=1}^n\mu_{\xi_{i-1}}^x(\xi_i(x))>0.$$

When $\w^s_{C_i}\supset\w^s_{C_{i-1}}$ we have that
$\xi_i=\xi_{i-1}$ and hence $\mu_{\xi_{i-1}}^x(\xi_i(x))>0$
trivially. So, let us assume that
$\w^s_{C_i}\subsetneq\w^s_{C_{i-1}}$ and hence $\eta_i>\eta_{i-1}$.
Since conditionals along $\w^i$ are atomic, by
Theorem \ref{graphargument}, we conclude that
$\mu_{\eta_{i-1}}^x(\eta_i(x))>0$ for a.e $x$.

Finally, we prove that $\mu_{\xi_{i-1}}^x(\xi_i(x))>0$ for a.e. $x$
using the following lemma with $X=\eta_{i-1}(x)$, $\eta=\eta_i$,
$\xi_1=\xi_{i-1}$ and $\xi_2=\xi_i$.
\begin{lemma}
Let $\xi_1, \xi_2$ and $\eta$ be three measurable partitions of a
Lebesgue space $X$. Let us assume that $\xi_2=\xi_1\vee\eta$.
Let $\mu$ be a measure in $X$. If there is $x$ such that
$\mu(\eta(x))>0$, then
$\mu_{\xi_1}^y(\xi_2(y))=\mu_{\xi_1}^y(\eta(x))$ and
$\mu_{\xi_1}^y(\eta(x))>0$ for $\mu$ a.e. $y$ in $\eta(x)$ and hence
$\mu_{\xi_1}^y(\xi_2(y))>0$ for a.e. $y\in\eta(x)$.
\end{lemma}
\begin{proof}
The first equality is trivial since $\xi_2(y)=\xi_1 (y\cap\eta(y))$
and $\eta(x)=\eta(y)$ for $y\in\eta(x)$. For the inequality, let
$D=\{y\in X:\mu_{\xi_1}^y(\eta(x))=0\}$. Observe that $D$ is
$\xi_1$-saturated. Let $B=D\cap\eta(x)$. We want to show that
$\mu(B)=0$. We have that
$$\mu(B)=\int\mu_{\xi_1}^y(B)d\mu=\int_D\mu_{\xi_1}^y(B)d\mu+\int_{D^c}\mu_{\xi_1}^y(B)d\mu.$$
The first integral in the right-hand side is $0$ since $B\subset \eta(x)$ and
$\mu_{\xi_1}^y(\eta(x))=0\}$ for $y\in D$. The second is zero since
for $y\in D^c$ we have that $\xi_1(y)\subset D^c$ and hence, since
$B\subset D$ we get that $\xi_1(y)\cap D=\emptyset$
\end{proof}

\subsection{Proof of Theorem~\ref{general}}
Theorem~\ref{general} is an immediate corollary of Theorem~\ref{nondifofentropy} and   the following general criterion of absolute continuity.

\begin{theorem}\label{absolutecontinuous}
Let $f:M\to M$ be a $\Ce$ diffeomorphism preserving an ergodic
measure $\mu$. Let $TM=E^u\oplus E^c\oplus E^s$ be the Oseledets
splitting associated to $\mu$. Let us assume that:
\begin{enumerate}
\item $E^c$ is tangent to a smooth foliation $\Or$, that $Df|E^c$ is
an isometry w.r.t. to the standard metric in $M$ and that
conditional measures along $\Or$ are Lebesgue measure,
\item $E^u=E_1\oplus\dots\oplus E_u$, $E^s=E_s\oplus\dots\oplus E_r$,
where $\chi_i<\chi_j$ if $i<j$,
\item each $E_i$ is tangent to an absolutely continuous Lyapunov foliation
$\w^i$ and conditional measures along $\w^i$ are absolutely
continuous w.r.t. Lebesgue for a.e. point.
\end{enumerate}
Then $\mu$ is absolutely continuous w.r.t. Lebesgue.
\end{theorem}
\begin{proof}
The proof reduces to see that conditional measure along stables and unstables are absolutely 
continuous. We shall argue by induction on the flag tangent to $E_1, E_1\oplus E_2,  \dots,
E_1\oplus\dots\oplus E_u=E^u$. So, let us call $V_i$ the "foliation" tangent to $E_1\oplus\dots
\oplus E_i$. That conditional measures along $V_1=\w^1$ are absolutely continuous is by 
assumption. Let us see that conditionals along $V_2$ is absolutely continuous and then the 
general step will follows as well. 

Let $R$ be the set of regular points. Take a regular point $x$ and 
a zero Lebesgue measure set $A\subset V_2(x)$. We want to see that $\mu_{V_2}^x(A)=0$ also. 
Taking $A\cap R$ we have that $\mu_{V_2}^x(A)=\mu_{V_2}^x(A\cap R)$ and $Leb_{V_2(x)}(A
\cap R)=0$. So, we may assume without loss of generality that $A$ consists of regular points (indeed 
arguing similarly we may also assume $A$ is inside some Pesin set if necessary). Now, since the 
foliation $\w^1=V_1$ is absolutely continuous and indeed it is also absolutely continuous when 
restricted top $V_2(x)$, we may saturate the set $A$ by $V_1$ leafs and get a set of $0$ 
$Leb_{V_2(x)}$-measure which is $V_1$-saturated and which contains $A$. Let us call this set by 
$B$ and let us see that $\mu_{V_2}^x(B)=0$.

Now, if $\mu_{V_2}^x(B)>0$ there should be a regular point $z\in V_2(x)$ such that 
$\mu^z_{\w^2}(B)>0$. But since $\mu^z_{\w^2}$ is equivalent to Lebesgue measure this will imply that $Leb_{\w^2(z)}(B)>0$ and again, absolute continuity of $V_1$ and the fact that $B$ is $V_1$ saturated will imply that $Leb_{V_2(x)}(B)>0$ which is a contradiction
.

\end{proof}

\section{Proof of Theorem \ref{toralcase} }\label{proofnil}
\subsection{General facts about entropy}\label{sbs:entropy}
We shall make use of the following standard facts.  Given measurable
partitions $\xi$, $\eta$ and $\zeta$ we have

\begin{enumerate}
\item $H(\xi\vee\eta|\zeta)=H(\xi|\zeta)+H(\eta|\zeta\vee\xi)$
\item If $\xi>\zeta$ then $H(\xi|\eta)\geq H(\zeta|\eta)$ and $H(\eta|\xi)\leq H(\eta|\zeta)$
\item If $\xi_n\uparrow\xi$ then $H(\xi_n|\eta)\uparrow H(\xi|\eta)$
\item If $\xi_n\downarrow\xi$ and $H(\xi_1|\eta)<\infty$ then $H(\xi_n|\eta)\downarrow H(\xi|\eta)$
\item If $\eta_n\uparrow\eta$ and $H(\xi|\eta_1)<\infty$ then $H(\xi|\eta_n)\downarrow H(\xi|\eta)$
\item If $\eta_n\downarrow\eta$ then $H(\xi|\eta_n)\uparrow H(\xi|\eta)$
\item For every $n\in\Z$, $h(\xi\vee\eta,T)= h(\xi\vee T^n\eta,T)$
\end{enumerate}

Given a measurable partition $\eta$ we denote
$\eta_T=\bigvee_{n\in\Z}T^n\eta$ and $\eta^+=\bigvee_{n=0}^\infty T^n\eta$.

\begin{lemma}\label{Newentropy}
Given two measurable partitions $\xi$ and $\eta$,
$$h(\xi\vee\eta,T)\leq h(\eta,T)+h(\xi\vee\eta_T,T).$$
\end{lemma}

\begin{proof}
First of all we have that
\begin{eqnarray*}
h(\xi\vee\eta,T)&=&H(T^{-1}\xi\vee T^{-1}\eta|\xi^+\vee\eta^+)\\
&=&H(T^{-1}\eta|\xi^+\vee\eta^+)+H(T^{-1}\xi|\xi^+\vee
T^{-1}\eta^+)
\end{eqnarray*}
Then we have that
\begin{eqnarray*}
h(\xi\vee T^{-n}\eta,T)&=&H(T^{-1-n}\eta|\xi^+\vee
T^{-n}\eta^+)+H(T^{-1}\xi|\xi^+\vee T^{-1-n}\eta^+)\\
&=& H(T^{-1}\eta|T^n\xi^+\vee
\eta^+)+H(T^{-1}\xi|\xi^+\vee T^{-1-n}\eta^+)\\
&\leq& H(T^{-1}\eta|\eta^+)+H(T^{-1}\xi|\xi^+\vee
T^{-1-n}\eta^+)
\end{eqnarray*}
On one hand, the last term   is bounded by
$H(T^{-1}\xi|\xi^+)=h(\xi,T)$ and also $\xi\vee
T^{-1-n}\eta^+\uparrow \xi\vee\eta_T$ hence this last term
$\downarrow H(T^{-1}\xi|\xi^+\vee\eta_T)$. On the other hand,
since $\eta_T$ is a $T$-invariant partition we get easily that
$h(\xi\vee\eta_T,T)=H(T^{-1}\xi|\xi^+\vee \eta_T)$ and
hence the inequality.
\end{proof}

Finally we have the standard formula: given an invariant partition
$\zeta$ (i.e. $p^{-1}\zeta=\zeta$) we have that

$$h(T)=h(T|\zeta)+\sup_{\xi}h(\xi\vee\zeta,T)$$

where $h(T|\zeta)=\sup_{\eta<\zeta}h(\eta,T).$

\begin{proposition}\label{preservationofentropy}
Let us consider $T:(X,\mu)\to (X,\mu)$, $S:(Y,\lambda)\to
(Y,\lambda)$ and assume that $S$ is a factor of $T$, via a measure preserving map
 $p:(X,\mu)\to(Y,\lambda)$. Let $\xi$ be a full entropy
partition for $T$ and $\eta$ a partition such that $\eta_S=\e=$
partition into points and $p^{-1}\eta<\xi$. Then
 $\eta$ is a full entropy partition for $S$.
\end{proposition}

\begin{proof}
Let us call $\zeta$ the partition into pre-images of $p$. We have
 on one hand
\begin{eqnarray*}
  h(T)=h(T|\zeta)+\sup_{\gamma}h(\gamma\vee\zeta,T)=h(S)+\sup_{\gamma}h(\gamma\vee\zeta,T)
\end{eqnarray*}
and on the other hand since
$(p^{-1}\eta)_T=p^{-1}(\eta_S)=p^{-1}\e=\zeta$:
\begin{eqnarray*}
  h(T)&=&h(\xi,T)=h(\xi\vee p^{-1}\eta,T)\leq
  h(p^{-1}\eta,T)+h(\xi\vee(p^{-1}\eta)_T,T)\\
  &=&h(\eta,S)+h(\xi\vee\zeta,T)
\end{eqnarray*}
where the inequality follows form Lemma~\ref{Newentropy}.
Thus  $h(\eta,S)=h(S)$  i.e. $\eta$ is a
full entropy partition. Observe also that $\xi$ is also a full
entropy partition for the fiber-entropy 
$\sup_{\gamma}h(\gamma\vee\zeta,T)$.
\end{proof}

\subsection{Matching of Lyapunov half-spaces}\label{lyapexp}
Here we assume $\a$ and $\ao$ are $\Z^k$ actions on an infranilmanifold $M$ as in Theorem
\ref{toralcase}. 

Unstable foliation $W^u_{\alpha_0(\m)}$ for an  element of   the  algebraic action $\ao$  is right
homogeneous.  Lyapunov foliations $W^i$  that are one-dimensional 
under our assumptions are intersections of unstable foliations of different elements and are projections of cosets of one-parameter subgroups in $N$.

  Let $p:M\to M$ be the
semiconjugacy between these actions and let $\mu$ be an ergodic large  measure
invariant by $\alpha$.
We want to prove that Weyl
chambers for both actions match. We do this in two steps. First we
prove a general result which requires no assumption  on the
linear action:

\begin{lemma}\label{preservationofweylchambers}
If $L$ is a Lyapunov hyperplane for $\alpha_0$ then $L$ is also a
Lyapunov hyperplane for $\alpha$ and Lyapunov half-spaces match.
\end{lemma}

\begin{remark} Observe that this lemma implies that the number of  non-proportional (and hence nonzero) Lyapunov exponents for the nonlinear action is greater or equal than the
number of   coarse Lyapunov distributions for the linear action. Then later
we prove that under our assumptions  Lyapunov hyperplanes for the nonlinear action  also  correspond to Lyapunov hyperplanes
 for the linear action.
\end{remark}

\begin{proof}[Proof of Lemma \ref{preservationofweylchambers}]
Assume by contradiction that there are two elements $\n,\m\in\Z^k$ on
different sides of $L$ but in the same Weyl chamber for $\alpha$.
Then we will have that $\w_{\alpha(\n)}^u(x)=\w_{\alpha(\m)}^u(x)$ but
$W_{\alpha_0(\n)}^u\neq W_{\alpha_0(\m)}^u$. Since the semiconjugacy maps unstable manifolds for $\a$ into unstable manifolds for $\ao$,  either for $\n$ or
$\m$ the following is true (we assume it is for $\n$):
$$p(\w_{\alpha(\n)}^u(x))\subset p(x)W_{\alpha_0(\n)}^u\cap
E_{\alpha_0(\m)}^u\subsetneq p(x)W_{\alpha_0(\n)}^u.$$

Now, we may take a full entropy increasing partition $\tilde\xi$ for
$\alpha(\n)$ subordinated to $\w_{\alpha(\n)}^u$ like the one built in
\cite{LY} and an increasing partition $\eta$ subordinated to
$yW_{\alpha_0(\n)}^u\cap W_{\alpha_0(\m)}^u$, again like in
\cite{LY}, and we may build them in such a way that
$p^{-1}\eta<\tilde\xi$.

Since the negative iterates of  $\alpha_0(\n)$ contract $W_{\alpha_0(\n)}^u\cap W_{\alpha_0(\m)}^u$
 we have that $\eta_{\alpha_0(\n)}=\e$ and hence
$(p^{-1}\eta)_{\alpha(\n)}=\zeta$. So we have that, using Proposition
\ref{preservationofentropy}, $\eta$ is a full entropy partition for
$\alpha_0(\n)$.

But since we are considering Lebesgue measure for $\alpha_0$,
$W_{\alpha_0(\n)}^u\cap W_{\alpha_0(\m)}^u\subsetneq
E_{\alpha_0(\n)}^u$ and $\eta$ is subordinated to
$yW_{\alpha_0(\n)}^u\cap W_{\alpha_0(\m)}^u$ we have that
$$h(\eta,\alpha_0(\n))<h(\alpha_0(\n))$$
which gives a contradiction.
\end{proof}

Now, we consider  the case at hand:
there are no proportional Lyapunov exponents  for the linear action, i.e. there are
 different exponents $\chi_1,\dots\chi_n$ and none of them are  proportional.
Thus in this case $\alpha$ also has $n$
different non-proportional Lyapunov exponents and that the Weyl
chambers coincides. In particular we have that there are positive
numbers $c_i$ such that $\tilde{\chi}_i=c_i\chi_i$ for $i=1,\dots
n$. So, we have already excluded zero exponents.

\begin{corollary}\label{lyapunov foliations}
For every $i=1,\dots, n$ there is a Lyapunov foliation $\w^i$
associated to $\tilde\chi_i$, such that the leaf $\w^i(x)$ is mapped
by the semiconjugacy $p$ into the corresponding coset
$W^i(p(x))$.
\end{corollary}

\subsection{Conclusion of the proof}
\begin{proposition}
Conditional measures along $\w^i$ are  equivalent to Lebesgue
a.e.
\end{proposition}

\begin{proof} We shall argue as in the proof of preservation of Weyl
chambers. Using Theorem~\ref{tech} and arguing  by contradiction  we may assume that conditional measures along
$\w^i$ are atomic a.e.  Then take $\t$ and $\s$ as in Theorem
\ref{graphargument} for the suspended action and we may take
$\s\in\Z^k$. We will use the same notation $\a$ and $\a_0$  for the suspended actions of  $\Rk$. Take now an $\a(-\s)$ increasing partition $\xi$
subordinated to $\w^s_{\at}(x)\subsetneq\w^s_{\as}(x)$. Then, since
$p(\w^s_{\at}(x))\subset p(x)W^s_{\aot}\subsetneq p(x)W^s_{\aos}$
we can build another partition $\eta$ subordinated to
$p(x)W^s_{\aot}$ such that $p^{-1}\eta<\xi$. Since by Theorem
\ref{graphargument} conditional measure along $\w^s_{\as}(x)$ is in
fact supported in $\w^s_{\at}(x)$ we have that $\xi$ is a full
entropy partition for $\a(-\s)$ and then by Proposition
\ref{preservationofentropy} $\eta$ should be also a full entropy
partition for $\a_0(-\s)$, but this is impossible since
$W^s_{\aot}\subsetneq W^s_{\aos}$.\end{proof}

Now  we can use Theorem
\ref{absolutecontinuous} and conclude that $\mu$ is an absolutely
continuous measure. This concludes the proof of Theorem~\ref{toralcase}(1).
 
Theorem~\ref{toralcase}(2)  follows exactly as in \cite{kk}. Or, more precisely,  it is  proven there 
using information that we already possess.  Namely \cite[Lemma 4.4]{kk} asserts that the semiconjugacy restricted to a.e. leaf of a Lyapunov foliation is a diffeomorphism. Hence it matches asymptotic rates of expansion/contaction along the foliations and thus  Lyapunov exponents.  \qed

\section{Proof of Theorem~\ref{toralcase2}}\label{prooftorus}

\subsection{Uniqueness}\label{SS:uniqueness} For the proof of uniqueness in Theorem \ref{toralcase2} we will use
the invariant  affine structures on stable manifolds of the action $\a$.   We shall prove that affine structures for
unstable manifolds of the nonlinear action $\a$ are intertwined by the
semiconjugacy with the standard affine structure of unstable spaces
for the linear action $\ao$. Notice that due to Theorem~\ref{toralcase}   no resonance condition  holds for $\a$.  Existence of these affine structures is guaranteed by the non-resonance condition, see \cite[Section 6.2]{kk-01}.

\begin{proposition}\label{prop-affine}
For every $\t\in\Z^k$ and on each leaf of $\w_{\at}^s$ there is a
unique smooth $\a$-invariant affine structure  together with a frame
such that for any regular point $x$ and $j$ such that $\chi_j(\t)<0$
the one-dimensional leaf $\w^j(x)$ is a coordinate line in
$\w_{\at}^s(x)$ and  for any regular point $z\in\w_{\at}^s(x)$ the
affine structure on  $\w^j(z)$ coincides with the restriction of the
affine structure on $\w_{\at}^s(x)$.
\end{proposition}

We will use additive notations for various invariant foliations associated with the action $\ao$.

\begin{proposition}\label{propconjaffine}
For almost every regular point $z$ the restriction of the
semiconjugacy $h$ to the leaf $\w_{\at}^s(z)$ is an affine bijection
between $\w_{\at}^s(z)$ and the hyperplane $p(z)+E_{\aot}^s$.
\end{proposition}
\begin{proof}
Take $z$ for which almost every  point of the leaf $\w_{\at}^s(z)$
with respect to the $s$-dimensional volume is regular. Since
conditional measures are equivalent to Lebesgue, the set of such
points is of full $\mu$ measure. Thus there is a dense subset of
$\w^{i}(z)$ where   leaves of $\w^{j}$ for all $j\neq i$ are
defined. By Proposition~\ref{prop-affine} any such manifold is a
part of  a corresponding line and its affine parameterization agrees
with the one coming from the affine structure  in $\w^{i}(z)$. But
we already know that the semiconjugacy on any leaf of $\w^{j}$ is
affine. Furthermore for each regular $y\in \w^{i}(z)$ the  manifold
$\w^{j}(y)$ cannot be just an interval  but must  be the whole line
in the affine structure. Thus we know that $h$ restricted to
$\w^{i}(z)$ is affine on   a dense set of lines parallel to each
coordinate direction. Hence it is  affine.
\end{proof}

\begin{lemma}\label{largeunstables}
For any point $z$ satisfying the assertion of
Proposition~\ref{propconjaffine}, the manifold $\w^{i}(z)$ is a 
complete manifold properly embedded into $\R^n$ 
and at a bounded distance from $E^i$.  Indeed, if we denote the 
semiconjugacy resticted to $\w^i(z)$ by $h_i^z:\w^i(z)\to h(z)+E^i$, 
then its inverse, $p_i^z:h(z)+E^i\to\w^i(z)$ is a proper diffeomorphism 
at a bounded distance from the inclusion. 

\end{lemma}
\begin{proof}
Proposition~~\ref{propconjaffine} implies this statement for any
compact   part of $\w^{i,+}(z)$. But since $h$ is a bounded distance
away from identity  for any sequence of points  on the $h(z)+E^{i}$
which goes to infinity the pre-images go to infinity too. The assertion 
of about the inverse of the semiconjugacy follows from the fact that the 
semiconjugacy is at a bounded distance from identity.
\end{proof}

Now we shall show how the Hopf argument applies in this case to get
uniqueness similar to what is done for instance in \cite{rhrhtu}. To this 
end we will need that for  any two given regular points $x_1, x_2$  
(possibly regular with respect to different large measures), the stable 
manifold of one intersects the unstable manifold of the other. This is done 
through an index argument.

\begin{lemma}\label{indexargument}
Let $E^i\subset\R^n$, $i=1,2$ be two subspaces such that  $E^1\oplus E^2=\R^n$. 
Let $p_i: E^i\to \R^n$, $i=1,2$ be two proper embeddings at a bounded 
distance from inclusion. Call $p_i(E^i)=W_i$, $i=1,2$. Then $W_1\cap W_2\neq\emptyset$. 
\end{lemma}
\begin{proof}
Let us assume by contradiction that $W_1\cap W_2=\emptyset$. 
Let $D_i$ be closed unit disks in $E_i$ and define for $0<t\leq1$, 
$$X_t:D_1\times D_2\to S^{n-1}\subset E^1\oplus E^2$$
 by 
 $$X_t(v_1,v_2)=\frac{p_1(v_1/t)-p_2(v_2/t)}{\|p_1(v_1/t)-p_2(v_2/t)\|}.$$ Observe that $X_t$ is well defined since the denominator is never $0$. Let us write $p_i(z)=z+\psi_i(z)$, we have that there is $C>0$ such that $\|\psi_i(z)\|\leq C$ for every $z\in E^i$. 

Let us see that as $t\to 0$ we have that $X_t$ restricted to $\partial (D_1\times D_2)$ converges uniformly to 
$$X_0(v_1,v_2)=\frac{v_1-v_2}{\|v_1-v_2\|}.$$
Indeed $$p_1(v_1/t)-p_2(v_2/t)=\frac{v_1-v_2}{t}+\psi_1(v_1/t)-\psi_2(v_2/t),$$
hence 
\begin{eqnarray*}
X_t(v_1,v_2)=\frac{p_1(v_1/t)-p_2(v_2/t)}{\|p_1(v_1/t)-p_2(v_2/t)\|}=\frac{v_1-v_2+t(\psi_1(v_1/t)-\psi_2(v_2/t))}{\|v_1-v_2+t(\psi_1(v_1/t)-\psi_2(v_2/t)))\|}
\end{eqnarray*}
since the $\psi_i$ are uniformly bounded and the denominator is is bounded away from 
zero when $(v_1,v_2)\in \partial (D_1\times D_2)$ and $t$ is small, we get $X_t\to X_0$ uniformly.

But then it is known that $X_0$ is a map of nonzero degree (it is a homeomorphism), while $X_1$ restricted to $\partial (D_1\times D_2)$ should have zero degree since it extendes to $D_1\times D_2$.

\end{proof}

Now, take $\mu_1$ and $\mu_2$ two ergodic large measures. Fix an
element of the action $f:=\an$ with all exponents nonzero. We shall
prove uniqueness using $f$. Let us call $G$ the set of points
satisfying the conclusion of Proposition \ref{propconjaffine}, we
have that $G$ has full measure for every large measure.

Take a continuous function $\phi$, we will prove that $\int\phi
d\mu_1=\int\phi d\mu_2$. Let us take a set $B_1\subset G$ of full
$\mu_1$ measure such that for $x$ en $B_1$,
$\phi^+(x)=\phi^-(x)=\int \phi d\mu_1$, here $\phi^+$ and $\phi^-$
denote forward and backward Birkhoff averages (with respect to $f$)
respectively. Similarly take a set $B_2\subset G$ of full $\mu_2$
measure where $\phi^+(x)=\phi^-(x)=\int \phi d\mu_2$ for $x\in B_2$.
Now take sets $A_i\subset B_i$ of full $\mu_i$ measure such that if
a point $x$ is in $A_i$ then $Leb_{W^u(x)}$ almost every point $y$
in $W^u(x)$ is in $B_i$. We have that $A_i$ have full measure by the
absolute continuity of the stable and unstable foliations.

We know that $\phi^+$ is constant on stable manifolds and $\phi^-$
is constant on unstable manifolds.

We now lift all the objects to the universal covering in order to
define holonomy maps in a more clear manner, we denote points in the
universal covering and in the manifold in the same manner and this
should not give any confusion. Take now two points $x_1\in A_1$ and
$x_2\in A_2$. By Lemmas \ref{largeunstables} and \ref{indexargument} 
we have that $W^s(x_1)\cap W^u(x_2)\neq\emptyset$,
but we do not know a priory how this intersection is. Now, the
semiconjugacy must send this intersection into the intersection of
$h(x_1)+E^s$ and $h(x_2)+E^u$ which is a point. By Proposition
\ref{propconjaffine}, we know that the semiconjugacy restricted to
$W^u(x_2)$ is one to one, hence this intersection is a point. Still
we do not know if this intersection is transversal, so we can not
follow the usual Hopf argument. In any case, since almost every
point in $W^u(x_1)$ is in $G$ we can define the holonomy map
$\pi:W^u(x_1)\to W^u(x_2)$ by $\pi(z)=W^s(z)\cap W^u(x_2)$. Observe
that $\pi$ is a priori only defined on a set of full Lebesgue
measure in $W^u(x_1)$. We want to prove that $\pi$ is absolutely
continuous but since the intersection defining $\pi$ is not
transversal a priori we can not follow the usual absolute continuity
proof. What we have is that the semiconjugacy restricted to
$W^u(x_i)$ is smooth, in fact it is affine with respect to the
affine structure and the semiconjugacy also conjugates the
holonomies, that is: if we define $Hol:h(x_1)+E^u\to h(x_2)+E^u$ as
we did with $\pi$ we have that $h\circ\pi=Hol\circ h$ for every
point in $W^u(x_1)$ where $\pi$ is defined. But $Hol$ is a smooth
map since $Hol$ is simply a translation also $h$ restricted to
$W^u(x_i)$ is smooth hence $\pi=h^{-1}\circ Hol\circ h$ coincides
a.e. with a smooth map and hence it is absolutely continuous. Now we
have that $B_1\cap W^u(x_1)$ has full Lebesgue measure in $W^u(x_1)$
and hence its image by $\pi$ has also full Lebesgue measure in
$W^u(x_2)$ and hence this image intersects $B_2$, that is, we can
take a point $a\in B_1$ whose stable manifold contains a point $b\in
B_2$ hence we have that $\int\phi
d\mu_1=\phi^+(x_1)=\phi^+(x_2)=\int\phi d\mu_2$ and we are done.

\subsection{Semiconjugacy and measurable isomorphism of  $\a$ and $\ao$}
Let us see that the semiconjugacy is one to one over a set of full measure. Let $R$ be the set of regular points satisfying the conclusion of Lemma \ref{largeunstables}. We shall see that the restriction of $h$ to $R$ is one to one. Let us fix a nonsingular element of the action. We already know that the restriction of $h$ to stable and unstable manifolds of regular points is a diffeomorphism. Take $x$ and $y$ in $R$ and assume by contradiction that $h(x)=h(y)=a$. By Lemmas \ref{largeunstables} and \ref{indexargument} we know that $\w^s(x)\cap\w^u(y)\neq\emptyset$. Take $z$ in this intersection. Then $h(z)\in (h(x)+E^s)\cap (h(y)+E^u)$ but since $h(x)=h(y)=a$ this last intersection is $a$ and hence $h(z)=a$. Now, injectivity along stable manifolds gives a contradiction since $z\in\w^s(x)$ and $h(z)=a=h(x)$. Since the image of $R$ 
has full measure, we get that the restriction of $h$ to $R$ gives a measurable isomorphism between $\a$ and $\ao$ and thus we finish the proof of Theorem \ref{toralcase2}.

\section{Proof of Theorem~\ref{thm-cocycles}}\label{s: cocycles}
 We first need to describe properly  the classes of cocycles considered in Theorem~\ref{thm-cocycles}. Let us fix a small positive number $\e$ and consider Pesin sets  $\Re_{\e}^l$ as defined in \eqref{Pesinset}.  
 
 Let us consider a Lyapunov Riemannian metric  on each Lyapunov distribution defined on the set of full measure $\Re_{\e}=\bigcup_l\Re_{\e}^l$. It is defined  similarly to \eqref{eqLyapunovmetric} with summation over $\Zk$ instead of integration. By \cite[Proposition 5.3]{KKRH} this metric is H\"older continuous on each $\Re_{\e}^l$. Now consider a system of neighborhoods  $P_\e (x)$, sometimes called {\em Pesin boxes}, of points in  $\Re_{\e}$ whose size depends on $l$ and slowly oscillates with the action,
similarly to the function $K_{\e}$ from Proposition~\ref{MET}. Using a local coordinate system from a fixed finite atlas project the Lyapunov metric from $T_x$ to the Pesin box around $x$ with constant coefficients. Thus we obtain a system of locally defined metrics.

\begin{definition}A cocycle $\beta$ defined on $\Re_{\e}$ is called {\em Lyapunov H\"older} if for any $l, \,\, x\in \Re_{\e}^l$ $\beta$ is H\"older continuous on $\Re_{\e}^l\cap P_\e (x)$ with H\"older exponent and constant independent of $x$ and $l$.

Similarly we define Lyapunov smooth  cocycles by requiring  smoothness  along local stable manifolds of points in  $\Re_{\e}^l$ with uniform bounds on derivative with respect to a Lyapunov metric within Pesin boxes.
\end{definition}
Notice that by Proposition~\ref{prop-affine} the semi-conjugacy $h$ between $\a$ and the  linear  action $\ao$ is bijective on an increasing  sequence of compact Pesin sets as well on stable and unstable manifolds of points from those sets  with respect to all elements of the action $\a$. The strategy of the proof is to  use these bijections to construct cocycles  over $\ao$ and then use the method of \cite{KNT}. 

Take the image of a  Pesin set $\mathcal P$ under the semi-conjugacy. If a   solution of the coboundary equation exists then   along the stable manifold $\w$ of any  element of the action    is given by the familiar telescoping sum see e.g. \cite[Proof of Theorem 3.1]{KNT}. This implies in particular that the solution
(transfer function) is Lyapunov H\"older or Lyapunov  smooth if the cocycle has  one of those properties.

By the absolute continuity $\w\cap\mathcal P$ has large conditional measure in $\w$ and the union of our Pesin sets has full conditional measure.  Now one considers periodic cycles anchored at points of the Pesin sets. Any two successive points in such a cycle
lie on a one-dimensional Lyapunov line and any three successive points lie in a stable manifold of some element. The last statement follows from the TNS condition that is weaker than our strongly simple property. One can simply consider the situation after the semi-conjugacy, as a cocycle over the linear action.  Arguing as  in \cite{KNT} we deduce that   solution can be constructed consistently  from a single typical point to the union of Pesin sets which has full measure. Since the semi-conjugacy is bijective on a full measure set and is smooth along almost every stable manifold the solution can be brought back,   and, as we pointed out,  is  then
 Lyapunov H\"older or Lyapunov smooth.\qed

\begin{remark}In the absence of semi-conjugacy  but  assuming  strongly simple and no resonance conditions one can still extend  the solution along Lyapunov lines but due to the ``holes'' in the union of Pesin sets the argument works only locally. This leads to the following statement.

Let $\mu$ be a measure as in  Theorem~\ref{general}.
The spaces of  classes of Lyapunov H\"older
(corr. Lyapunov smooth) cocycles with respect to cohomology with Lyapunov H\"older  (corr. Lyapunov smooth) transfer functions   are finite dimensional.

Even in the absence of such holes   the solution can be constructed on the universal cover but cannot in general be projected to the original manifold since  the possibility of the action preserving a non-trivial homology class  cannot  be excluded.
\end{remark}

\section{Beyond the strongly simple case}\label{s: beyond}
\subsection{Summary}
We can  tentatively claim  partial generalizations of the some results of this paper in the presence of  multiple or positively proportional (but not negatively proportional) exponents.  We can also outline the limits of  applicability for our methods and formulate plausible conjectures. 

One should consider separately the general case of hyperbolic measures 
for smooth actions  as in Section~\ref{sbs:stronglysimple}  and large measures  for  actions on tori and nilmanifolds with  hyperbolic homotopy data as in Section~\ref{sbs:tori-nil}. 
At the level of linear algebra three effects may appear separately on in combinations:
\begin{enumerate}
\item {\em Negatively proportional exponents}. Our methods that are essentially geometric are  not suitable for this situation. The main problem with using the linear algebra of Lyapunov exponents is that in the representative symplectic case the picture of Lyapunov hyperplanes and Weyl chambers is the same as for the product of rank one actions where rigidity does not take place. Thus  there is not much hope for developing a general theory along the lines of  \cite{KKRH}.  

Even for algebraic actions on  a torus measure rigidity is established  by different methods  that take into account global Diophantine properties of stable foliations \cite{EL}. Another approach    can be developed based on an unpublished preprint of J. Feldman and M. Smorodinsky from the early 1990s. Finding a non-uniform version of these arguments  is a serious albeit not a hopeless challenge. 
\smallskip

\noindent Hyperbolic measures without  negatively proportional  Lyapunov exponents are called {\em  totally non-symplectic} (TNS).\medskip

\item {\em  Multiple exponents.} The first  central step  of our approach is ``freezing'' the action in question 
along the walls of Weyl chambers.  Notice that  for   linear (and hence algebraic)  actions this is possible  in the semisimple case, i.e. in the absence of Jordan blocks.  For actions on  tori  and nilmanifolds assuming that the algebraic model (the homotopy data)  is semisimple helps.  In the presence of jordan blocks  the situation is less hopeful.
\medskip

\item {\em Simple positively proportional exponents.} This case (assuming that no other effects appear)  is the most hopeful  
and is discussed in more detail below. The key issue here is understanding 
 resonances and invariant geometric structures that appear on coarse Lyapunov foliations. 
\end{enumerate}

\subsection{Hyperbolic measures for actions on  manifolds} The main difficulty  here is  that   vanishing of a Lyapunov exponent  does not guarantee that along Lyapunov foliations (even if the exponent is simple and if those exist)  on a set of large measure 
the distances remain bounded.  The technical devise that allows to overcome this problem  in the strongly simple case  is the synchronizing time change   described in Section~\ref{sbs:synchronizing}.  This is easily modified   to obtain   bounded  growth estimates  like in Proposition~\ref{timechange} for any  one  Lyapunov  direction, e.g.  the fastest for which Lyapunov distribution is  integrable. However, in general fin general    time change would be different  in general so  no simultaneous ``freezing'' is possible. 
This problem looks fundamental  and  probably cannot be overcome  within the  usual rank  $\ge 2$  assumption. 

But in fact in our argument  synchronization of one exponent is achieved along the whole Lyapunov hyperplane.  If the rank is $\ge 3$  that  {\em simultaneous} synchronization of two exponents can  be achieved along a codimension two subspace and so on.   This of course  is under the assumption that exponents are simple. Thus the following statement holds: 
\medskip

\noindent {\em Simultaneous synchronization of all  proportional exponents is possible if their number   does not exceed the rank of the action minus one.}\medskip

Now one considers  an invariant geometric structure on the coarse Lyapunov foliation.  In the absence of double resonances  this structure is flat  affine and one can show that Lyapunov distributions integrate to  foliations into lines with respect to this structure.  The critical $\pi$-partition argument holds in this case and allows  to show that conditional measures on the coarse Lyapunov foliation   are  supported on an affine subspaces and invariant under  transitive groups of
affine transformations on those subspaces. Hence those conditional measures are either atomic or absolutely continuous on  smooth submanifolds of the  leaves  of the coarse Lyapunov foliations. 
\medskip

The arguments outlined above       lead to the proof of a {\em generalization of Theorem~\ref{stronglysimple} for TNS  actions   with simple  positively proportional exponents, no double resonances   if the number of exponents proportional to a given one   does not exceed the rank of the action minus one.} \medskip

Detailed proofs will appear in a subsequent paper.

The  case with double resonances  is  somewhat more complicated because for the 
slow directions there are no unique curves tangent tot the slow directions 
similar to likes in the affine case.  Instead there  are  some parametric families of such curves like parabolas in the case of $2:1$ resonance. If one can prove that  Lyapunov  distributions integrate to  certain families of such curves  the rest of the argument should be similar to the non-resonance case.

An extension of Theorem~\ref{general}  looks more problematic. The problem is that the full entropy assumption  does not  catch   contributions coming from different positively proportional exponent. One should look for a 
   an appropriate ``high entropy''  assumptions  that would  lead   to the assertion that conditional measure along  the coarse Lyapunov foliation is absolutely continuous. After that absolute continuity of the measure can be established, similarly to the proof of Theorem~\ref{general}.

\subsection{Actions on tori and nil-manifolds}
As was mentioned above, our methods are restricted to the TNS case so me make this assumption for the algebraic action $\ao$. To be able to carry out the ``freezing'' argument
  we also  need to avoid  Jordan blocks for the   action $\ao$, i.e. to assume that its linear part is semi-simple (diagonalizable over $\C$).
  
Then the action along the Lyapunov hyperplanes is an isometry. The main issue  is to prove that there are no new Lyapunov hyperplanes for the  action $\a$. 
So far we  can prove this is certain special cases, e.g.  If new Lyapunov exponents for $\a$  not proportional to those of $\ao$ (and hence new Lyapunov hyperplanes) appear, 
corresponding Lyapunov  foliations must collapse under the semi-conjugacy.

Entropy considerations like in Section~\ref{sbs:entropy} provide for that:  collapsing of certain directions 
leads to entropy deficit although   arguments become more involved.  

After that one can follow the  general line of arguments
in Sections~\ref{proofnil} and \ref{prooftorus} to obtain {\em an extension of Theorem~\ref{toralcase2} to the TNS non-resonance case.} A particular  case  where    double exponents are allowed due to existence of complex eigenvalues for $\ao$  is  announced in \cite{KKRH-ERA}.
Detailed proofs will appear in a subsequent paper.

Resonances  both for $\ao$ and  for $\a$ represent  an additional difficulty but basically   one should prove intertwining of geometric structures and hence smoothness of the semi-conjugacy along the coarse Lyapunov foliations. 
Thus one can  formulate  desired outcome as follows. 
\begin{conjecture} Let $\ao$ be  a totally non-symplectic  $\Z^k$ action  by automorphisms of an infranilmanifold and $\a$ be an action with homotopy data  $\ao$. Then every large invariant measure for $\a$ is absolutely continuos and has the same Lyapunov characteristic exponents as $\ao$.  
\end{conjecture}

\bibliographystyle{alpha}

\begin{thebibliography}{999999}

\bibitem{BP} L. Barreira and Ya. Pesin,  {\it Lyapunov exponents and smooth ergodic theory}, University lecture series, {\bf 23} AMS, 2002.

\bibitem{BPbook}  L. Barreira and Ya. Pesin, {\em Nonuniform Hyperbolicity: Dynamics of Systems with Nonzero Lyapunov Exponents}, Encyclopedia of Mathematics and
Its Applications, {\bf 115} Cambridge University Press, 2007.




\bibitem{DK-II} D. Damjanovic and A. Katok,{\em  Local Rigidity of Partially Hyperbolic Actions. II. The geometric method and restrictions of Weyl chamber flows on $SL(n, R)/\Gamma$ },  \\    www.math.psu.edu/katok$\_$a/papers.html

\bibitem{EL} M. Einsiedler and E. Lindenstrauss,  \emph{Rigidity properties of
{${\mathbb
  {Z}}\sp d$}-actions on tori and solenoids}, Electron. Res. Announc. Amer.
  Math. Soc. \textbf{9} (2003), 99--110.


\bibitem{hu} H. Hu {\em Some ergodic properties of commuting diffeomorphisms}, Ergodic Theory Dynam. Systems {\bf 13} (1993), no. 1, 73--100. 

\bibitem{kk-01} B. Kalinin and A. Katok, {\em Invariant measures for actions of higher
rank abelian groups}, Proc. Symp. Pure Math,  {\bf 69}, (2001),
593-637.

\bibitem{kk} B. Kalinin and A. Katok, {\em Measure rigidity beyond uniform hyperbolicity:
Invariant Measures for Cartan actions on Tori},  Journal of Modern
Dynamics, {\bf 1} N1 (2007), 123--146.




\bibitem{KKRH-ERA}B. Kalinin, A. Katok and  F. Rodriguez Hertz, {\em  New Progress in Nonuniform Measure  and Cocycle Rigidity} Electronic Research Announcements in Mathematical Sciences, {\bf 15}, (2008), 79--92.

\bibitem{KKRH}  B. Kalinin, A. Katok and F. Rodriguez Hertz, {\em Nonuniform Measure Rigidity},\\  http://www.math.psu.edu/katok$\_$a/pub/KKRH.pdf

 \bibitem{KNT} A. Katok, V. Nitica and A.T\"or\"ok, {\em Non-Abelian cohomology of abelian Anosov actions},  Ergod. Th. \& Dynam. Syst., {\bf 20}, (2000), 259--288.


\bibitem{KRH} A. Katok and  F. Rodriguez Hertz, {\em Uniqueness of large invariant measures for $\Zk$  actions with Cartan homotopy data}, Journal of Modern Dynamics, {\bf 1}, N2, (2007), 287--300.


\bibitem{KS1} A. Katok and R. J. Spatzier  {\em First cohomology of Anosov actions of higher rank abelian groups and applications to rigidity},   Publ. Math. IHES, {\bf79}, (1994), 131-156.

\bibitem{KS2}  A. Katok and R. J. Spatzier, {\em Subelliptic estimates of polynomial differential operators and applications to rigidity of abelian actions,} Math. Res. Letters, {\bf1}, (1994), 193-202.

\bibitem{LY} F. Ledrappier, L.-S. Young, {\it  The metric entropy of diffeomorphisms. I. Characterization of measures satisfying Pesin's  entropy formula}, Ann. Math.
{\bf 122}, no. 3, 509--539.

\bibitem{LY2} F. Ledrappier, L.-S. Young. {\em The metric entropy of diffeomorphisms.
Part II: relations between entropy, exponents and dimension}, Ann.
Math.
 {\bf 122} (1985), no. 3, 540--574.

\bibitem{rhrhtu} F. Rodriguez Hertz, M. Rodriguez Hertz, A. Tahzibi, R. Ures, \emph{A criterion for ergodicity of non-uniformly hyperbolic diffeomorphisms.} {\it ERA-MS}  {\bf 14}, (2007) 74--81.


\end{thebibliography}

\end{document}